\documentclass[11pt,reqno,a4paper]{amsart}

\usepackage{todonotes}
\usepackage{amsmath}
\usepackage{amsfonts}
\usepackage{mathrsfs}  
\usepackage{amssymb}
\usepackage{amsthm}
\usepackage{hyperref}
\usepackage{graphicx}
\usepackage[noadjust]{cite}
\usepackage{caption}
\usepackage{dutchcal}
\usepackage{bbm}
\usepackage{tikz}
\usetikzlibrary{decorations.pathreplacing}
\usepackage[T1]{fontenc}
\usepackage[utf8]{inputenc}
\usepackage{mathtools}

\usepackage[top=3.5cm, bottom=3.5cm, left=2.7cm, right=2.7cm, headsep=0.3in]{geometry}

\usepackage{fancyhdr}

\theoremstyle{plain}
\newtheorem{theorem}{Theorem}[section]

\newtheorem{corollary}[theorem]{Corollary}
\newtheorem{lemma}[theorem]{Lemma}
\newtheorem{remark}[theorem]{Remark}
\newtheorem{proposition}[theorem]{Proposition}

\newcommand{\RR} {\mathbb R}

\newcommand{\NN} {\mathbb N}

\newcommand{\NNN}{\mathcal{N}}

\newcommand{\pa} {\partial}
\newcommand{\Cal} {\mathcal}

\newcommand{\beq} {\begin{equation}}
\newcommand{\eeq} {\end{equation}}

\newcommand{\capacity}{\operatorname{cap}}

\numberwithin{equation}{section}

\begin{document}

\title{
Nodal sets of Laplace eigenfunctions under small perturbations}
\author{Mayukh Mukherjee and Soumyajit Saha}
\address{IIT Bombay, Powai, Maharashtra 400076, India}
\email{mathmukherjee@gmail.com}
\email{ssaha@math.iitb.ac.in}
\maketitle
\begin{abstract}
	We study the stability properties of nodal sets of Laplace eigenfunctions on compact manifolds under specific small perturbations. We prove that nodal sets are fairly stable if such perturbations are relatively small, more formally, supported at a sub-wavelength scale.  
	We do not need any generic assumption on the topology of the nodal sets or the simplicity of the Laplace spectrum. As an indirect application, we are able to show that a certain ``Payne property'' concerning the second nodal 
	set remains stable under controlled perturbations of the domain.
\end{abstract}
\section{Introduction}

In this note, we are largely interested in certain aspects of the stability of nodal sets of Laplace eigenfunctions under small perturbations. 
We start by recalling that the Laplace-Beltrami operator\footnote{We use the analyst's sign convention, that is, $-\Delta$ is positive semidefinite.} on any compact manifold (with or without boundary) has a discrete spectrum 
$$ 
\lambda_0 = 0 \leq \lambda_1 \leq \dots \leq \lambda_k \leq \dots \nearrow \infty,
$$
with corresponding smooth real-valued eigenfunctions $\varphi_k$ that satisfy
$$
-\Delta \varphi_k = \lambda_k \varphi_k.
$$
Recall that $\NNN_\varphi = \{ x \in M : \varphi (x) = 0\}$ denotes the nodal set (or vanishing set) of the eigenfunction $\varphi$, and any connected component of $M \setminus \NNN_\varphi$ is known as a nodal domain of the eigenfunction $\varphi$. These are domains where the eigenfunction is not sign-changing (this follows from the maximum principle). Recall further that the nodal set is the union of a $(n - 1)$-dimensional smooth hypersurface and a ``singular set'' that is countably $(n - 2)$-rectifiable (\cite{HS}). 
 As a notational remark, when two quantities $X$ and $Y$ satisfy $X \leq c_1 Y$ ($X \geq c_2Y$) for constants $c_1, c_2$ dependent on the geometry $(M, g)$, we write $X \lesssim_{(M, g)} Y$ (respectively $X \gtrsim_{(M, g)} Y$). 
	Unless otherwise mentioned, 
	these constants will in particular be independent of eigenvalues $\lambda$. Throughout the text, the quantity $\frac{1}{\sqrt{\lambda}}$ is referred to as the wavelength and any quantity (e.g. distance) is said to be of sub-wavelength (super-wavelength) order if it is $\lesssim_{(M, g)} \frac{1}{\sqrt{\lambda}}$ (respectively $\gtrsim_{(M, g)} \frac{1}{\sqrt{\lambda}}$). 

In a certain sense, the study of eigenfunctions of the Laplace-Beltrami operator is the analogue of Fourier analysis in the setting of compact Riemannian manifolds. Recall that the Laplace eigenequation is the standing state for a variety of partial differential equations modelling physical phenomena like 
heat diffusion, wave propagation or Schr\"{o}dinger problems. 
Below, we note down this well-known method of ``Fourier synthesis'':
    \begin{eqnarray}
    \text{Heat equation} & (\pa_t - \Delta)u = 0 & u(t, x) = e^{-\lambda t}\varphi(x)\\
    \text{Wave equation } & (\pa^2_t - \Delta)u = 0 & u(t, x) = e^{i\sqrt{\lambda} t}\varphi(x)\\
    \text{Schr\"{o}dinger equation} & (i\pa_t - \Delta)u = 0 & u(t, x) = e^{i\lambda t}\varphi(x)
    \end{eqnarray}
Further, note that $u(t, x) = e^{\sqrt{\lambda }t}\varphi(x)$ solves the harmonic equation $(\pa^2_t + \Delta)u = 0$ on $\RR \times M$. 

A motivational perspective of the study of Laplace eigenfunctions then comes from quantum mechanics (via the correspondence with Schr\"{o}dinger operators), where the $ L^2 $-normalized eigenfunctions induce a probability density $ \varphi^2(x) dx $, i.e., the probability density of a particle of energy $ \lambda $ to be at $ x \in M $.
Another physical (real-life) motivation of studying the eigenfunctions, dated back to the late 18th century, is based on the acoustics experiments done by E. Chladni which were in turn inspired from the observations of R. Hooke in the 17th century. The experiments consist of drawing a bow over a piece of metal plate whose surface is lightly covered with sand. When resonating, the plate is divided into regions that vibrate in opposite directions causing the sand to accumulate on parts with no vibration. 
The study of the patterns formed by these sand particles (Chladni figures) were of great interest which led to the study of nodal sets and nodal domains. Among other things, one interesting 
line of speculation is to see how the Chladni figures change after we distort our plate ``slightly''. 
A variant of the above question is the following: suppose hypothetically one could construct a vibrating plate which is perfectly symmetrical. Due to the symmetry, the eigenfunctions might appear with some multiplicity, and then a very natural question is: the sand pattern will assume the nodal set of which eigenfunction once we hit the corresponding frequency\footnote{The heuristic that symmetry implies spectral multiplicity seems to be well-known in the community. If the symmetry group of the plate is non-commutative, then some one can bring in some representation theoretic arguments to justify the above mentioned multiplicity. 
For a detailed proof, see \cite{Ta}.}? Unfortunately, we do not know the answer to this question. However, since physical objects are never perfectly symmetrical (or one can perturb the plate slightly), 
considerations based on transversality (see \cite{U}) dictate that almost all such plates have simple spectrum, and it is interesting to observe the pattern on the perturbed plates to try to conjecture which pattern the sand particles would assume on the hypothetical symmetrical plate. But this entails achieving a stability result on the nodal sets under slight perturbations.


We begin by listing below some questions/speculations that we find interesting and try to give some motivations behind them.
\begin{enumerate}
    \item[1.] It is known that many aspects of nodal sets of Laplace eigenfunctions are rather unstable under perturbation, which normally disallows perturbative techniques (like normalized Ricci flow and related geometric flows etc.) in the study of nodal geometry. 
    Even $C^\infty$-convergence of eigenfunctions is not strong enough to give convergence of nodal sets in the Hausdorff sense (or any other appropriate sense). The problem is, in the limit the nodal set can become topologically non-generic (for example, developing a node). As a trivial example, consider the functions $x^2y^2 + \epsilon^2$, which converge to $x^2y^2$ and develop a non-trivial nodal set in the limit. 
    This instability is observed for many metric properties like $\Cal{H}^{n - 1}(\NNN_\varphi)$, inner radius estimates of nodal domains etc whose known estimates are far better for real analytic metrics than for smooth metrics. However, one is inclined to ask the question that if the perturbations are ``small enough'', are there certain ``soft'' properties of the nodal set that are still reasonably stable? 
    From heuristic considerations, one further speculates that  perturbations which do not disturb the nodal set too much, if they exist, should be supported at a subwavelength scale.
    \item[2.] A variant of the above question is the following. Consider two compact orientable surfaces $M$ and $N$. Can one define a metric on the connected sum $M \# N$ such that the nodal sets of first $l$ eigenfunctions are completely contained in $M$, where $l$ can be arbitrarily large? In other words, the low energy  nodal sets must not ``see'' $N$. 
    \item[3.] Thirdly, there are whole classes of results about nodal sets which are true under rather restrictive metric assumptions. As an example, there are results for low energy eigenfunctions which are true only under the assumption of convexity of the domain. For example, see Theorem 1.1 of \cite{M}, or the main results of \cite{J, GJ}. As a preliminary case, one can then begin to speculate if such results are still true on slightly perturbed (non-convex) domains if perturbations are small enough. More interestingly, one can consider a variant of the above questions in the spirit of homogenization problems, that is, are the aforementioned results true on ``perforated'' domains with the ``size'' of the perforations in some appropriate sense (eg., capacity or volume) being smaller than some critical threshold? 
    \item[4.] A variant of the question in 2. above has been answered in \cite{K} under assumptions on $M$, and also topological assumptions on the nodal set, and for only the first non-trivial eigenfunction. Could we remove these assumptions, and prove a result  on general manifolds $M$ and $N$, for any eigenfunction, and without any topological restrictions on the nodal set? 
\end{enumerate}
We are able to give 
reasonable answers to several of the above questions  (see Subsection \ref{subsec:overview} below). We wish to draw particular attention to the methods of \cite{T}, which were adapted by \cite{K} (and also further work in \cite{T1, AT}), which in turn were further modified by \cite{EP-S} (who used these ideas in conjunction with ideas introduced in \cite{CdV}). We also refer the reader to \cite{A, AC}, which were precursors to the work of \cite{T}. Also, a significant number of ideas are (sometimes implicitly) contained in the paper \cite{RT}. Our methods are largely inspired from (and follow in certain details) the above references. Note that the topological restrictions on the nodal set (namely, the nodal set being a submanifold of $M$) in \cite{K, EP-S} etc. are actually generic in the sense of \cite{U}. 
\subsection{Overview of the paper}\label{subsec:overview} In Section \ref{sec:2}, we outline the general constructions involved in the ``connected sums'' of two closed manifolds or two domains (with boundary and Dirichlet conditions), and the associated perturbative shrinking process and state our main results. In Section \ref{sec:3}, we outline some results of \cite{T} and \cite{K}, and point out some modified settings under which such results still continue to hold. We also prove our main result Theorem \ref{thm:main_thm} in this section, which addresses Questions 
2 and 4 as mentioned above. In Section \ref{sec:4}, we go over some applications of Theorem \ref{thm:main_thm} and also some related ideas in the spirit of \cite{O, RT}. Here we outline some applications to the well-known Payne problem of the nodal set of the second Dirichlet eigenfunction intersecting (or not intersecting) the boundary. We give two results to this end, which are recorded in Propositions \ref{cor:Melas} and \ref{cor:del_hol_cap}. These are aimed at addressing Question 3 above. Finally, we quickly outline the proof of Corollary \ref{cor:Lewy}, which gives an extension of the main Theorem of \cite{K}.
\section{Statements of the results}\label{sec:2} 
\subsection{Some general constructions}
Let $(M_1, g_1)$ and $(M_2, g_2)$ be two closed connected oriented Riemannian manifolds of dimension $ n\geq 2$. 
Following Takahashi in \cite{T}, we can construct the metric 
\begin{equation}\label{eq:Riem_met_def}
    \Tilde{g_{\epsilon}}= 
\left\{
	\begin{array}{ll}
	g_1 & \mbox{on } M_1(\epsilon) \\
		\epsilon^2 g_2 & \mbox{on } M_2(1)
	\end{array}
\right.
\end{equation}
on $$M:= M_1(\epsilon) \cup_{\phi_{\epsilon}} M_2(1),$$
where
$$M_1(\epsilon)=M_1 \setminus B(x_1, \epsilon), \hspace{5pt} x_1\in M_1,$$
$$M_2(1)= M_2 \setminus B(x_2, 1), \hspace{5pt} x_2\in M_2,$$
and $$\phi_{\epsilon}: \pa M_1(\epsilon)\to \pa M_2(1) \hspace{5pt} \text{ defined as } x\mapsto x/\epsilon$$
is the ``attaching map''. Observe that it is fine to take $\epsilon < \text{inj}M_1$, and by scaling $g_2$, one can assume without loss of generality that $\text{inj}M_2 > 1$. It is clear that for $M$ to have the ``natural'' orientation, the attaching map $x \mapsto x/\epsilon$ has to be orientation reversing. \newline
\noindent Also, define
$$M_2(\epsilon):=  (M_2(1), \epsilon^2g_2).$$
Note that the metric $\Tilde{g_{\epsilon}}$ is only piecewise smooth, but one can still define Laplacians etc. via piecewise quadratic forms (for example, see Definition 2.2-2.4 and Lemma 2.5 in \cite{T}). As in \cite{BU}, it is known that Laplace spectrum varies continuously with respect to $C^0$ topology. So, one can perturb $\tilde{g}_\epsilon$ slightly to get a smooth metric $g_\epsilon$ whose spectrum is arbitrarily close to that of $\tilde{g}_\epsilon$ (see also Step 4 of Section 2 of \cite{EP-S}). So, one can instead work with smooth metrics $g_\epsilon$ such that $g_\epsilon|_{M_1(\epsilon/2)} = g_1$ which are arbitrarily close to $\tilde{g}_\epsilon$ in $C^0$-topology and with arbitrarily close spectrum. 

With that in place, now define
$$M_{\epsilon}= (M, g_{\epsilon}).$$
\subsection{Statements of the results}
We begin by stating our main theorem:
\begin{theorem}\label{thm:main_thm}
Consider the connected sum $M_1\#M_2$ of two closed oriented Riemannian manifolds $M$ and $N$ of dimension $n \geq 2$. Given  a natural number $m$, we can define a one-parameter family of (smooth) metrics $g_\epsilon$ on $M_1 \# M_2$ such that $g_\epsilon|_{M_1(\epsilon)} = g_1|_{M_1(\epsilon)}$, and the nodal set of any eigenfunction $\varphi_l$ of the Laplace-Beltrami operator on $(M_1 \# M_2, g_\epsilon)$ lies completely inside $M_1$ for all $\epsilon < c(m)$ and $l \leq m$, where $c(m)$ is a positive constant dependent only on $m$ and the geometry of $(M_1, g_1)$. 
\end{theorem}
\begin{remark}\label{rem:2.2}
It will be clear from the proof that once $m > m_0$ (where $m_0$ is a  large enough constant depending on the initial geometric data), the allowable perturbations must be  below the length scale $\epsilon 
\leq cm^{-1/n}$, where the constant $c$ also depends on the initial geometric data (for a more precise description, see the discussion following the proof of Theorem \ref{thm:main_thm} below).
\end{remark}
Recall that Courant's nodal domain theorem says that an eigenfunction corresponding to $\lambda_j$ can have at most $j$ nodal domains. A series of results starting from work by Pleijel seeks to prove (in different settings) that there are only finitely many Courant sharp eigenvalues. Turning things around, one can ask the question: what about higher eigenfunctions which have very few nodal domains? Of course, by orthogonality, such higher eigenfunctions have must have at least $2$ nodal domains. As a corollary of Theorem \ref{thm:main_thm} above (and classical work by \cite{L}), we are able to show that on any closed surface of genus $\gamma \geq 1$, one can design a metric such that it has arbitrarily many eigenfunctions with $2$ nodal domains. In fact, we show slightly more:
\begin{corollary}\label{cor:Lewy}
Given a closed surface $M$ of genus $\gamma \geq 1$ and a natural number $m \in \NN$, one can find a metric $g = g(m)$ on $M$ such that there are $\mu(m)$ (respectively, $\nu(m)$) eigenfunctions $\varphi_j, j \leq m$ 
which have $ 2$ (respectively $3$) nodal domains. Here $\mu(m)$ (respectively, $\nu(m)$) is the number of 
odd (respectively, even) integers $k$ such that $k(k + 1) \leq \lambda_m(S^2, g_{round})$. 
\end{corollary}
We remark that Corollary \ref{cor:Lewy} gives an extension of the Main Theorem of \cite{K} (see p 2405 and Problem 1.2 of \cite{K}).

Our proof can also be mimicked to cover other analogous cases. For example, consider a domain $\Omega \subset \RR^n$ and a one-parameter family of deformations $\Omega_\epsilon$, where $\Omega_\epsilon$ can be written as a disjoint union $\Omega_1^\epsilon \sqcup \Omega_2^\epsilon$, where 
\begin{itemize}
    \item $\Omega_1^\epsilon := \Omega \setminus B(x_1, \epsilon)$, for some $x_1 \in \pa \Omega$.
    \item $\Omega_2 \subset \RR^n$ is another domain, and $\Omega_2^\epsilon$ is obtained from $\Omega_2$ by scaling $g|_{\Omega_2^\epsilon} = \epsilon^2 g|_{\Omega_2}$ (please see diagram below). 
\end{itemize}   

Observe that we have started with a Euclidean domain which can be expressed as the disjoint union of two Euclidean domains similar to a ``tinkertoy'' or ``ball-and-socket'' joint. This 
is designed precisely to ensure that each deformation $\Omega_\epsilon$ is a Euclidean domain. In particular, one cannot blindly repeat the ``attachment''  construction as in  (\ref{eq:Riem_met_def}) above, as the attached metric will not be Euclidean (or even flat) in general.

 \vspace{2mm}
    $$\includegraphics[height=3.7cm]{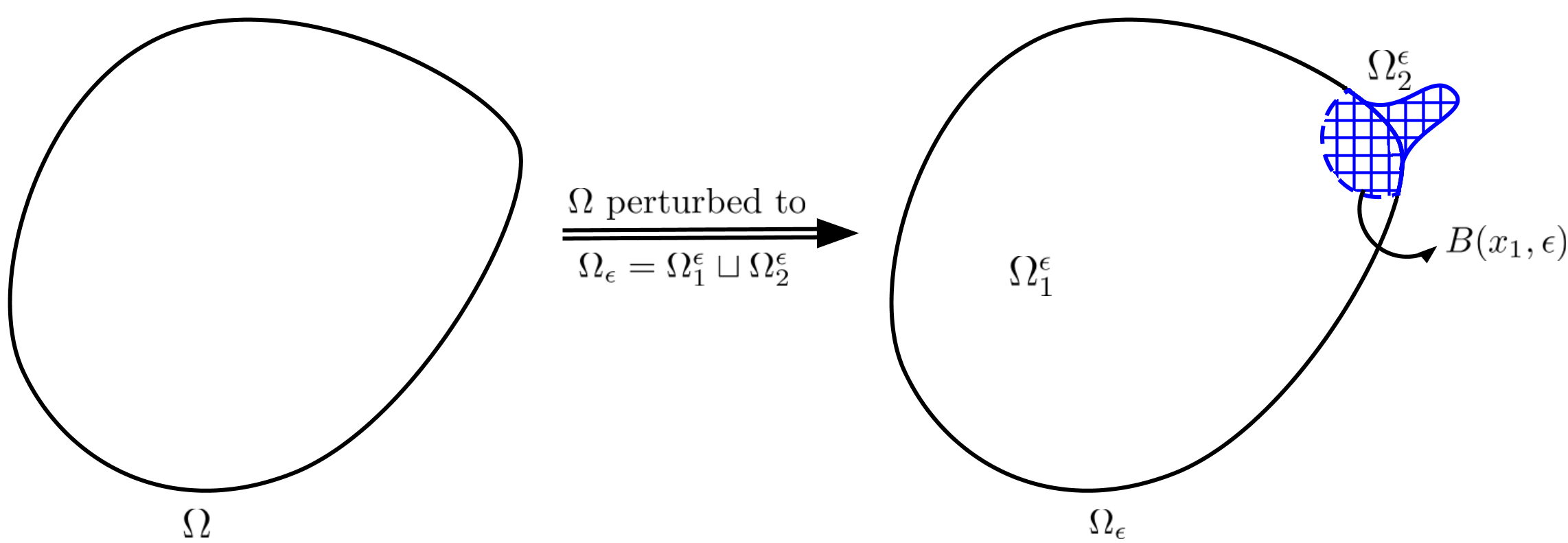}$$

Our findings will also extend to this setting when one considers the Laplacian with Dirichlet boundary conditions. 
In particular, this shows that conclusions like Theorem 1.1 of \cite{M} also hold true under certain non-convex perturbations and also with perforated domains (see Section \ref{sec:4}, Propositions \ref{cor:Melas} and \ref{cor:del_hol_cap} below). As another example of an application: \cite{J} proves that the diameter of the nodal set of the second Dirichlet eigenfunction of a convex planar domain $\Omega$ is $\leq C \text{inradius }\Omega,$ where $C$ is an absolute constant. The methods of the proof also illustrates that this nodal line lies near the middle of $R$, where $R \subset \Omega$ is an inscribed rectangle with the least Dirichlet eigenvalue $\lambda_1(R)$. Our methods will also extend such results to domains with small enough non-convex perturbations. Lastly, a further modification of the above arguments should also give a perturbed version of Corollary 2.2 of \cite{Fr}, though we do not carry this out.
\section{Proof of Theorem \ref{thm:main_thm}}\label{sec:3}
\subsection{Preliminary results}
Before beginning with the proof proper, we outline some preliminaries needed for the proof.

In Theorem 1.1 of \cite{T}, it is shown that: 
\begin{theorem}[Takahashi]\label{thm:Tak_1.1}
For all $k = 0, 1, \cdots,$ we have 
\begin{equation}
    \lim_{\epsilon \to 0} \lambda_k(M, g_{\epsilon}) = \lambda_k(M_1, g_1).
\end{equation}
\end{theorem}

We also introduce the following notation: let $f_k^\epsilon$ denote the $k$th Laplace eigenfunction of $M_\epsilon$ and $f_k$ denote the $k$th Laplace eigenfunction of $M_1$. With that in place, we quote the following result, which is a combination of Theorem 3.4 and Corollary 3.5 of \cite{K} (see also \cite{BU}). 
\begin{theorem}[Komendarczyk]\label{theo: eigenfunction convergence}
For each $k$, we have  
the following $C^\infty-$convergence of eigenfunctions $f_k^{\epsilon}\in C^{\infty}(M_{\epsilon})$:
\begin{equation}\label{eq:3.2}
    \lim_{\epsilon\to 0} f_k^{\epsilon}= f_k \hspace{10pt} \text{on compact subsets of } M_1\setminus \{x_1\}
\end{equation}
where all the eigenfunctions involved are $L^2$-normalized,
$$
\|f_k^{\epsilon}\|_{L^2(M_{\epsilon})} = \|f_k \|_{L^2(M_1)} = 1.
$$
\end{theorem}
It should be noted that \cite{K} proves Theorem \ref{theo: eigenfunction convergence} assuming that the Laplace spectrum of $M_1$ is simple, which is a generic condition. However, a close inspection of the proof in \cite{K} together with the ideas in the proof of Theorem 1.1 of \cite{T} reveals that Komerdarczyk's result will continue to hold even when the Laplace spectrum of $M_1$ is not simple. In particular, we refer to ideas on pages 206-207 of \cite{T}, which can be easily modified to our setting to prove that an orthonormal family $f^\epsilon_k$ of converging eigenfunctions continues to remain orthonormal in the limit, that is, if  $\lim_{\epsilon\to 0} f_k^{\epsilon}= f_k $, then $f_k$ is still an orthonormal family. That, coupled with the convergence of spectrum given by Theorem 1.1 of \cite{T}, proves our claim.


We also state the following well-known property of nodal sets, but we also include a proof for the sake of completeness.
\begin{proposition}\label{thm:nod_wav_dense}
If $\varphi$ satisfies $-\Delta \varphi = \lambda \varphi$ on a compact manifold (with or without boundary) $M$, then the nodal set $N_\varphi$ is asymptotically wavelength dense. More formally, there exists a positive constant $C = C(M, g)$ and $\lambda_0 = \lambda_0(M, g)$ such that any ball $B(x, C\lambda^{-1/2})$ intersects $N_\varphi$ for all $\lambda \geq \lambda_0$. 
\end{proposition}

\begin{proof}
Observe that $u(t, x) := e^{\sqrt{\lambda} t}\varphi (x)$ is a harmonic function in $[-1, 1] \times M$. If there is a ball $B(y, R) \subset M$ where $\varphi$ does not vanish, by Harnack inequality, there exists a constant $C := C(R, M, g)$ such that $u(t_1, x) \leq C u(t_2, z)$ for all $(t_1, x), (t_2, z) \in [-1, 1] \times B(y, R)$. Since $u(t, x)/u(-t, x) = e^{2\sqrt{\lambda}t}$, choose $t = \lambda^{-1/2}\log C_1$, such that $C_1^2 > C_2$. This violates the previous assumption, and shows that $B(y, C\lambda^{-1/2})$ contains a nodal point. 
\end{proof}

Before we begin the proof of Theorem \ref{thm:main_thm}, we first comment that though Takahashi proves Theorem \ref{thm:Tak_1.1} for closed manifolds, it can be checked such a result is also true for compact Euclidean domains with 
smooth boundary, with only the obvious modifications in the proof. The same remark applies to Theorem \ref{theo: eigenfunction convergence}. In other words, consider the one-parameter family of Euclidean domains $\Omega_\epsilon$ formed via the process 
described below Corollary \ref{cor:Lewy}. Then, the following variant of  Theorems \ref{thm:Tak_1.1} and \ref{theo: eigenfunction convergence} is true. 
\begin{theorem}\label{thm:dom_eigen_conv}
For all $k = 0, 1, \cdots,$ we have 
\begin{equation}\label{eq:eigenval_conv}
    \lim_{\epsilon \to 0} \lambda_k(\Omega_\epsilon) = \lambda_k(\Omega).
\end{equation}
For each $k$, we have  
the following $C^\infty-$convergence of eigenfunctions $f_k^{\epsilon}\in C^{\infty}(\Omega_\epsilon)$:
\begin{equation}
    \lim_{\epsilon\to 0} f_k^{\epsilon}= f_k \hspace{10pt} \text{on compact subsets of } \Omega\setminus \{x_1\}
\end{equation}
where all the eigenfunctions involved are $L^2$-normalized,
$$
\|f_k^{\epsilon}\|_{L^2(\Omega_{\epsilon})} = \|f_k \|_{L^2(\Omega)} = 1.
$$
\end{theorem}
One can check that the proofs given in \cite{T} and \cite{K} would work in this setting with only the obvious modifications. We skip the details to avoid repetition.

\subsection{Proof of Theorem \ref{thm:main_thm}.}
Given a Laplace eigenfunction $f_l$ (respectively, $f^\epsilon_l$), we use $\NNN_l(M_1)$ (respectively, $\NNN_l (M_\epsilon)$) to denote the respective nodal sets (sometimes, when it is clear from the context, we will suppress the index $l$ with minor abuse of notation). We assume that the nodal set $\NNN(M_1)$ is at a geodesic distance $d>2\epsilon_0$ from the gluing ball $B(x_1, \epsilon) \subset B(x_1, \epsilon_0)$ where $2\epsilon_0$ is a positive number chosen such that it is smaller than the injectivity radius at the point $x_1$. The nodal sets $\NNN(M_1)$ and $\NNN(M_{\epsilon})$ can be compared only on the common subset $M_1(\epsilon_0)$. 


What we want to prove is that for some $\epsilon>0$ the nodal set $\NNN(M_{\epsilon})$ belongs entirely inside the common domain $M_1(\epsilon_0)$. This would give the existence of a smooth metric on $M_{\epsilon}$ whose nodal domain is completely contained in $M_1$.


Given a number $m$, for each $k \leq m$ we will now choose a point $x_1 \in M_1$ (depending on $k$) such that $B(x_1, 2\epsilon_0)$ does not intersect the nodal set $\NNN_k(M_1)$ for the eigenfunctions {$ f_k$}. We first discuss how such an $\epsilon_0$ can be chosen depending only on $m$, and in particular, independently of $k$. Let $\Omega_{\lambda_l}$ be a nodal domain of an eigenfunction $\varphi_{\lambda_l}$ corresponding to the eigenvalue $\lambda_l$. Then, for all $l \in \NN$, a geodesic ball of radius $c\lambda_l^{- \frac{n - 1}{4} - \frac{1}{2n}}$ can be completely embedded inside $\Omega_{\lambda_l}$, where $c$ is a constant dependent only on $(M_1, g_1)$ (see \cite{Ma}, and further refinements in \cite{GM, GM1}). If the manifold has a real analytic metric or is $2$-dimensional, the above estimate can be further improved (see \cite{G}, \cite{Ma1}). Thus, it is clear that given a natural number $m$, one can choose $x_1 \in M_1$ and $\epsilon_0 := \epsilon_0(m)$ small enough such that $B(x_1, 3\epsilon_0)$ does not intersect the nodal set for any $f_l, l\leq m$. 

Another remark that is of utmost importance: it is a generic condition that Laplace eigenfunctions occur with multiplicity $1$. In that case, everything we have said in the above paragraph goes through. An issue appears, however, when multiple eigenfunctions exist for the same eigenvalue. In that case, even after fixing $k$, it is unclear that one is able to choose a point $x_1$ that avoids the nodal sets of all the linear combinations of all the  eigenfunctions corresponding to the eigenvalue $\lambda_k$. In fact, one is inclined to speculate that when one varies across all the linear combinations, the corresponding nodal sets might ``sweep out'' the whole of $M$\footnote{Here we have used the word ``sweepout'' in an intuitive sense. However, for linear combinations of eigenfunctions (not necessarily for same eigenvalue), a more precise mathematical formulation exist for this intuitive picture. See \cite{BBH} for such a formulation.}. However, even then the statement of Theorem \ref{thm:main_thm} remains true by varying the point $x_1$ (even when $k$ is fixed). This is because the connected sum of two connected oriented manifolds is well-defined, independent of the choice of the balls along whose boundary the identification is made. More formally, the  definition of the connected sum of two oriented $n$-manifolds $M$ and $N$  begins by considering two open $n$-balls $B_M$ in $M$, $B_N$ in $N$, and gluing the manifolds $M\setminus B_M$ and $N \setminus B_N$ along their boundary (an $(n-1)$-sphere) by an orientation-reversing homeomorphism.  The construction should depend a priori on the choice of these balls, but it turns out that it does not. For surfaces, as is the case in \cite{K}, this can be checked by hand, but for higher dimensions this is a deep topological result, and we refer the reader to Theorem 5.5  of \cite{P} for the differentiable category. This is fine with us, as the statement of Theorem \ref{thm:main_thm} makes no reference to the point $x_1$. 

Now the choice of $B(x_1, \epsilon_0)$ seems pretty clear. 
Given an eigenfunction $f_k$ ($k \leq m$), choose $x_1$ which is located ``deeply inside'' some nodal domain $\Omega_k$, and choose $\epsilon_0 = \min \{.3c\lambda_m^{-\frac{n - 1}{4} - \frac{1}{2n}}, \text{inj}(x_1)\}$ (inj($x_1$) denotes the injectivity radius at $x_1$), such that the ball $B(x_1, \epsilon_0)$ is completely contained inside the nodal domain $\Omega_k$. 
The connected sums we obtain by choosing $x_1$ and $\epsilon_0$ 
are all diffeomorphic and the corresponding Riemannian manifolds obtained via the construction in (\ref{eq:Riem_met_def}) are isometric. 

We now continue our proof with the following

\begin{lemma}\label{lem:lim_M_1}
Let $\NNN_k(M_{\epsilon})$ and $\NNN_k(M_1)$ denote the nodal sets corresponding to 
$f^\epsilon_k$ and $f_k$ respectively. Consider a sequence of points $\{x_i\}$ such that for each $i$, $x_i\in \NNN_k(M_{\epsilon_i})\cap M_1(\epsilon_0)$. If the limit $x$ of $\{x_i\}$ exists, then $x\in \NNN_k(M_1)$.
\end{lemma}
\begin{proof}
Recall that $f_k$ and $f_k^{\epsilon_i}$ are the $k^{th}$ eigenfunctions of the Laplace-Beltrami operators of $M_1$ and $M_{\epsilon_i}$ respectively, and 
following further notations from \cite{T, K}, we define 
$$
f_k^{1,\epsilon_i}:= f_k^{\epsilon_i}|_{M_1(\epsilon)}, f_k^{2,\epsilon_i}:= f_k^{\epsilon_i}|_{M_2(\epsilon)}.$$
 Using Theorem \ref{theo: eigenfunction convergence}, we apply the convergence  
 $$f_k^{\epsilon_i}|_{M_1(\epsilon_0)}\to f_k|_{M_1(\epsilon_0)} \text{ in  } C^0 (M_1(\epsilon_0))$$
 to obtain,
 $$|f_k(x_i)|= |f_k(x_i)- f_k^{1,\epsilon_i}(x_i)|\leq \|f_k- f_k^{1, \epsilon_i}\|_{C^0 (M_1(\epsilon_0))} \to 0 \hspace{5pt} \text{ as } i\to \infty.$$
  Now, $f_k$ being a continuous function, $x_i\to x$ implies $f_k(x_i)\to f_k(x)$. Therefore, $f_k(x)=0$ i.e. $x\in \NNN_k(M_1)$. 
\end{proof}

Now, we come to our next claim:
\begin{lemma}\label{lem:nod_set_eventually}
The nodal sets $\NNN_k(M_{\epsilon_i})$ are eventually in $M_1(\epsilon_0)$ i.e., we can find an index $p$ such that for all $i>p$, $\NNN_k(M_{\epsilon_i}) \subset M_1({\epsilon_0})$.
\end{lemma}


\begin{proof}
The proof will be divided into two cases:\newline
Case I: There exists an infinite sequence of points $\{x_i\}\subset \NNN_k(M_{\epsilon_i})\cap M_1(\epsilon_0)$.\newline
Case II: For small enough $\epsilon_i$, $\NNN_k(M_{\epsilon_i})\cap M_1(\epsilon_0) = \emptyset$. 
\newline
We first tackle Case I. 


Recall that we would like to show that $\NNN_k(M_{\epsilon_i}) \subset M_1(\epsilon_0)$ for $\epsilon_i$ small enough. If not, suppose there exists a subsequence $\{j\}\subset \{i\}$ such that $\NNN_k(M_{\epsilon_j})\nsubseteq M_1(\epsilon_0)$ i.e., $\NNN_k(M_{\epsilon_j})$ is contained in both $M_1(\epsilon_0)$ and $M_{\epsilon_j}\setminus M_1(\epsilon_0)$. It can happen in the following three ways:
\begin{enumerate}
    \item[$(C):$] The nodal set passes from $M_1(\epsilon_0)$ to $M_{\epsilon_j}\setminus M_1(\epsilon_0)$ in a continuous path and no other disconnected component of $\NNN_k(M_{\epsilon_j})$ is contained in $M_{\epsilon_j}\setminus M_1(\epsilon_0)$ (see Figure 1).
    
    \vspace{2mm}
    
    $$\includegraphics[height=3.7cm]{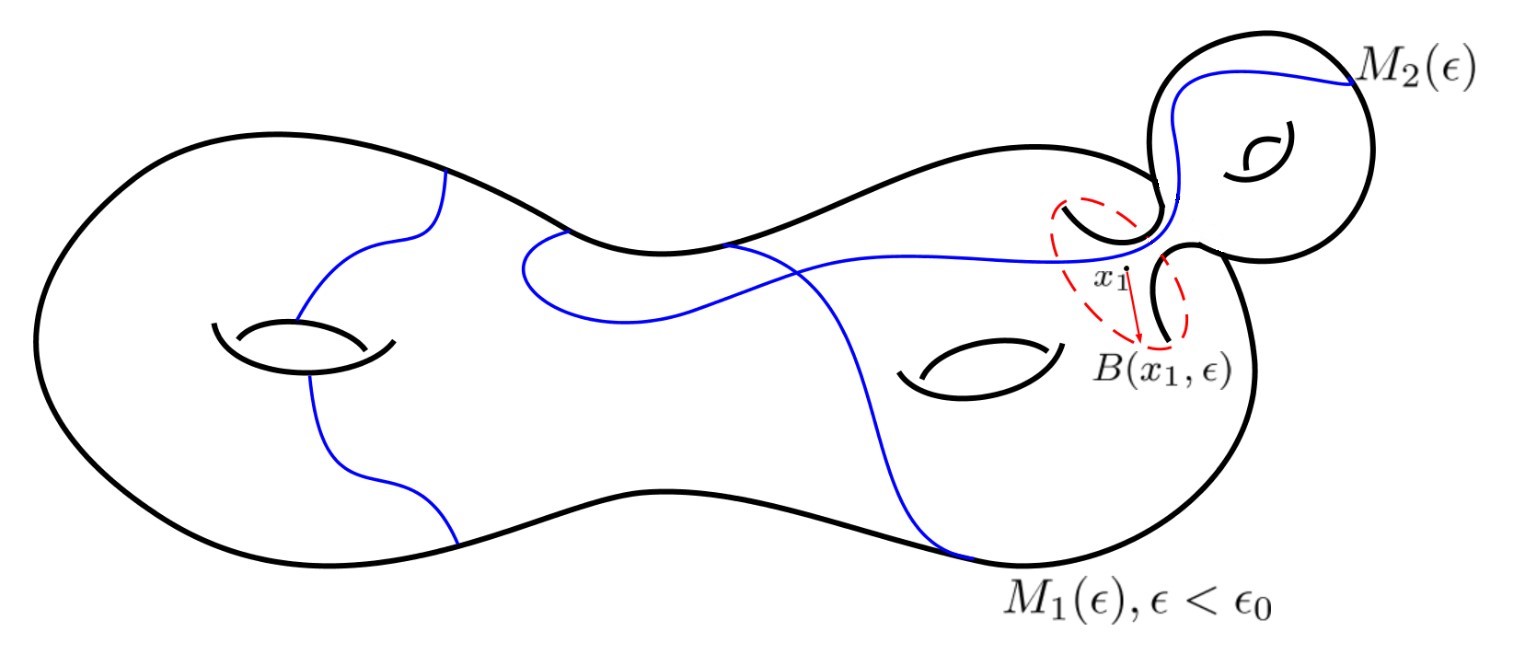}$$

    \begin{center}
     
        {\small Figure 1}
    \end{center}
    
    \vspace{5mm}
    \item[$(D1):$] A component of the nodal set passes from $M_1(\epsilon_0)$ to $M_{\epsilon_j}\setminus M_1(\epsilon_0)$ in a continuous path but some other disconnected component of $\NNN_k(M_{\epsilon_j})$ is also contained in $M_{\epsilon_j}\setminus M_1(\epsilon_0)$ (see Figure 2).
    
        \vspace{2mm}
        
    $$\includegraphics[height=3.7cm]{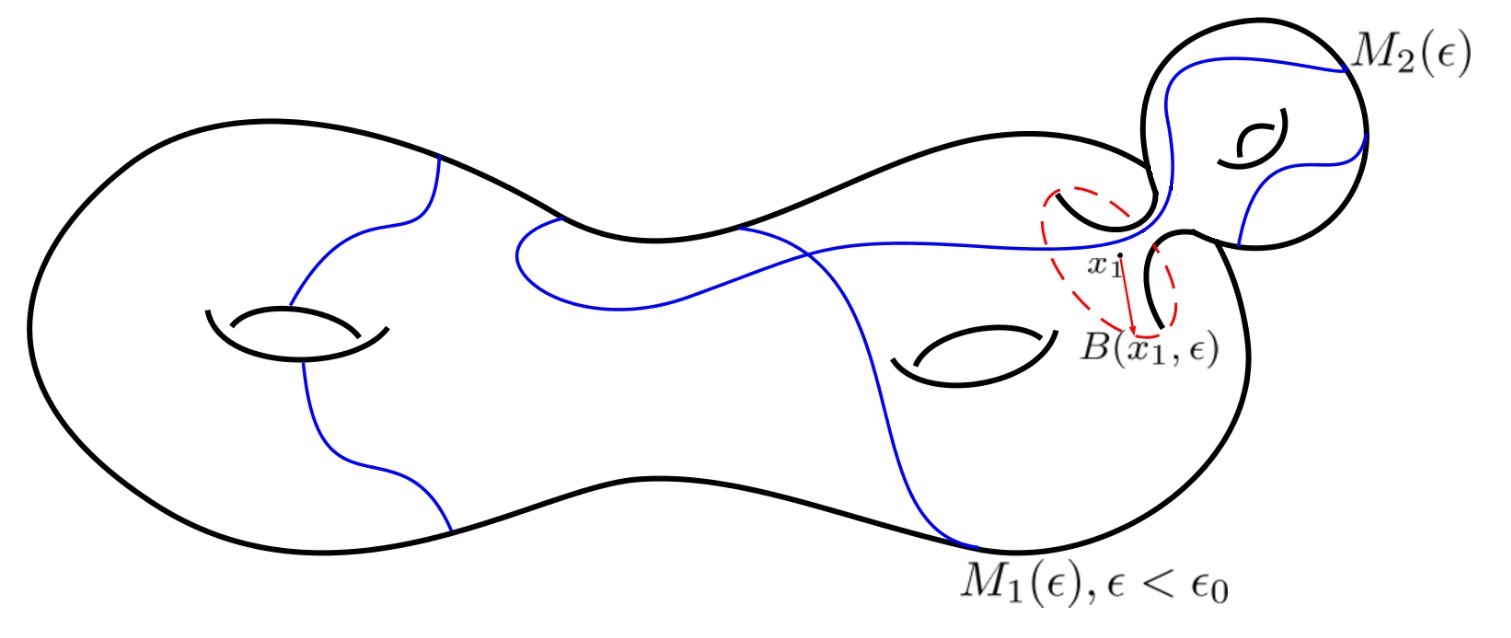}$$
    
    \begin{center}
        {\small Figure 2}
    \end{center}
    \vspace{5mm}
    \item[$(D2):$] There is no continuous path as a part of the nodal set from $M_1(\epsilon_0)$ to $M_{\epsilon_j}\setminus M_1(\epsilon_0)$ but the nodal set of $\NNN_k(M_{\epsilon_j})$ contains at least one  disconnected component in $M_{\epsilon_j}\setminus M_1(\epsilon_0)$ (see Figure 3). 
    
        \vspace{2mm}
        
    $$\includegraphics[height=3.7cm]{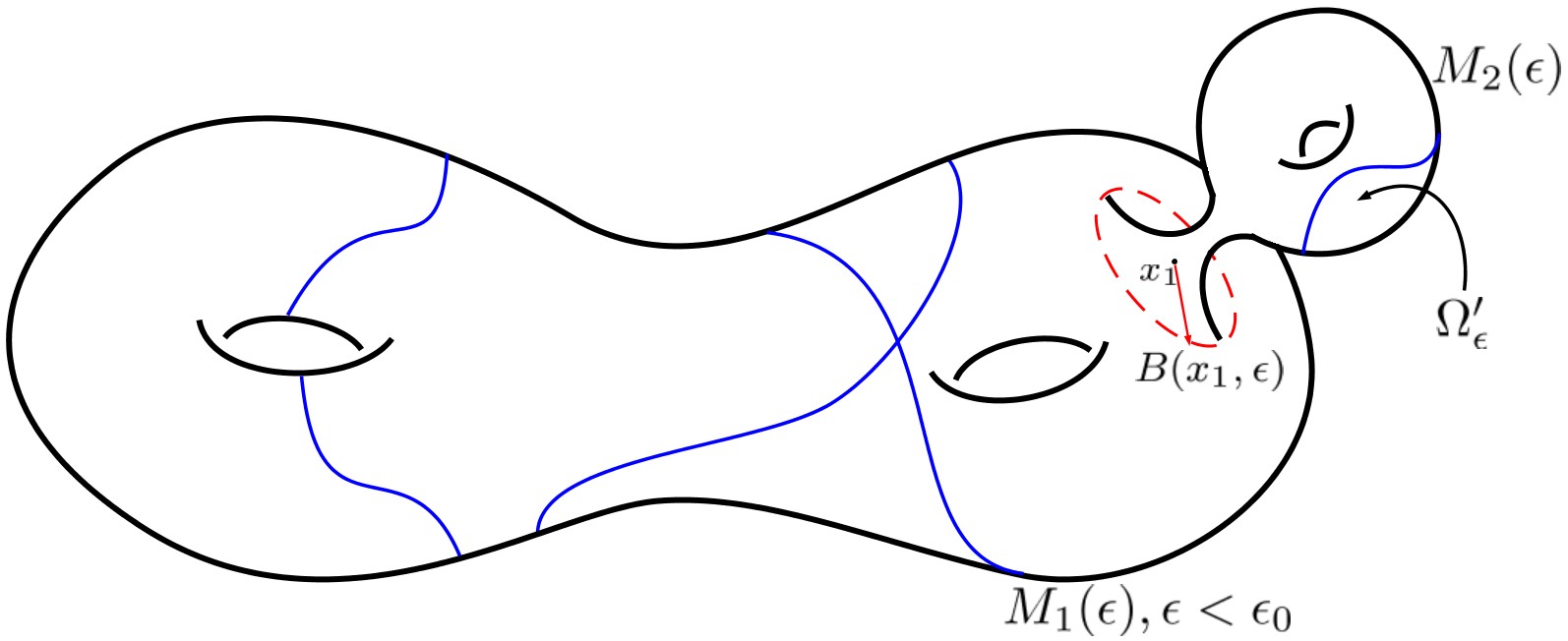}$$
     \begin{center}
        {\small Figure 3}
    \end{center}
    \vspace{5mm}
\end{enumerate}

We would like to show that as $\epsilon_j\to 0$ any disconnected component of $\NNN_k(M_{\epsilon_j})$ lies inside $M_1(\epsilon_0)$ i.e. cases $(D1)$ and $(D2)$ hold only for finitely many $j$'s.  Note that, it is enough to show that case $(D2)$ holds for finitely many $j$'s. 

If not, let there exist a subsequence $\{l\}\subset \{j\}$ such that $(D2)$ holds. 
Let $\Omega'_{\epsilon_l}$ 
denote a nodal domain completely contained in $M_2(\epsilon_l)$ 
(see Figure 3). We would like to show that  $\lambda_1(\Omega'_{\epsilon_l})\to \infty$ as $l\to \infty$ i.e. as $\epsilon_l\to 0$.

    


Note that as $\epsilon_l \to 0$, from the definition of the given metric, we have that $|\Omega'_{\epsilon_l}| \to 0$. 
Clearly, using the definition of the given metric, every nodal domain $\Omega'_{\epsilon_l}$ contained in $M_2(\epsilon_l)$ can be re-scaled to $\Omega_l$ (say) contained in $(M_2(1), g_2)$ which gives us the equality
\begin{equation}\label{eqtn: rescaled eigenvalue}
    \lambda_1(\Omega'_{\epsilon_l}) = \frac{1}{\epsilon_l^2}\lambda_1(\Omega_l).
\end{equation}
Also, for any $l$, using domain monotonicity we have
\begin{equation}\label{ineq: domain monotonicity}
    \lambda_1(\Omega_l) \geq \lambda_1(M_2(1)).
\end{equation}
Combining (\ref{eqtn: rescaled eigenvalue}) and (\ref{ineq: domain monotonicity}), we have that $\lambda_1(\Omega'_{\epsilon_l}) \to \infty$ as $\epsilon_l \to \infty$.


Note that,
\begin{equation}
    \lambda_1(\Omega'_{\epsilon_l})
    = \lambda_k(M_{\epsilon_l}).
\end{equation}

By Theorem \ref{thm:Tak_1.1}, we have that as $l\to \infty$ i.e. $\epsilon_l \to 0$, 
\begin{equation}\label{theo: Takahashi}
    \lambda_k(M_{\epsilon_l}) \to \lambda_k(M_1),
\end{equation}
which contradicts $\lambda_1(\Omega'_{\epsilon_l})\to \infty$. So, $(D2)$ cannot hold for infinitely many $\epsilon_j$'s.

From the above observation, if there exists a subsequence $\{j\}\subset \{i\}$ such that for some $y_j\in \NNN_k(M_{\epsilon_j})$, $y_j \notin M_1(\epsilon_0)$ then $\NNN_k(M_{\epsilon_j})$ satisfies condition $(C)$ 
except for finitely many $j$'s (which can be ignored). Then for each $j$, we get a continuous path from $x_j$ to $y_j$ 
contained in $\NNN_k(M_{\epsilon_j})$ whose one end is in $M_1(\epsilon_0)$ and the other end in $M_{\epsilon_j}\setminus M_1(\epsilon_0)$. Then there exists a point $z_j\in \NNN_k(M_{\epsilon_j})\cap \pa M_1(\epsilon_0)$ (see Figure 4). 

    
     \begin{center}
     \vspace{2mm}
     \includegraphics[height=3.7cm]{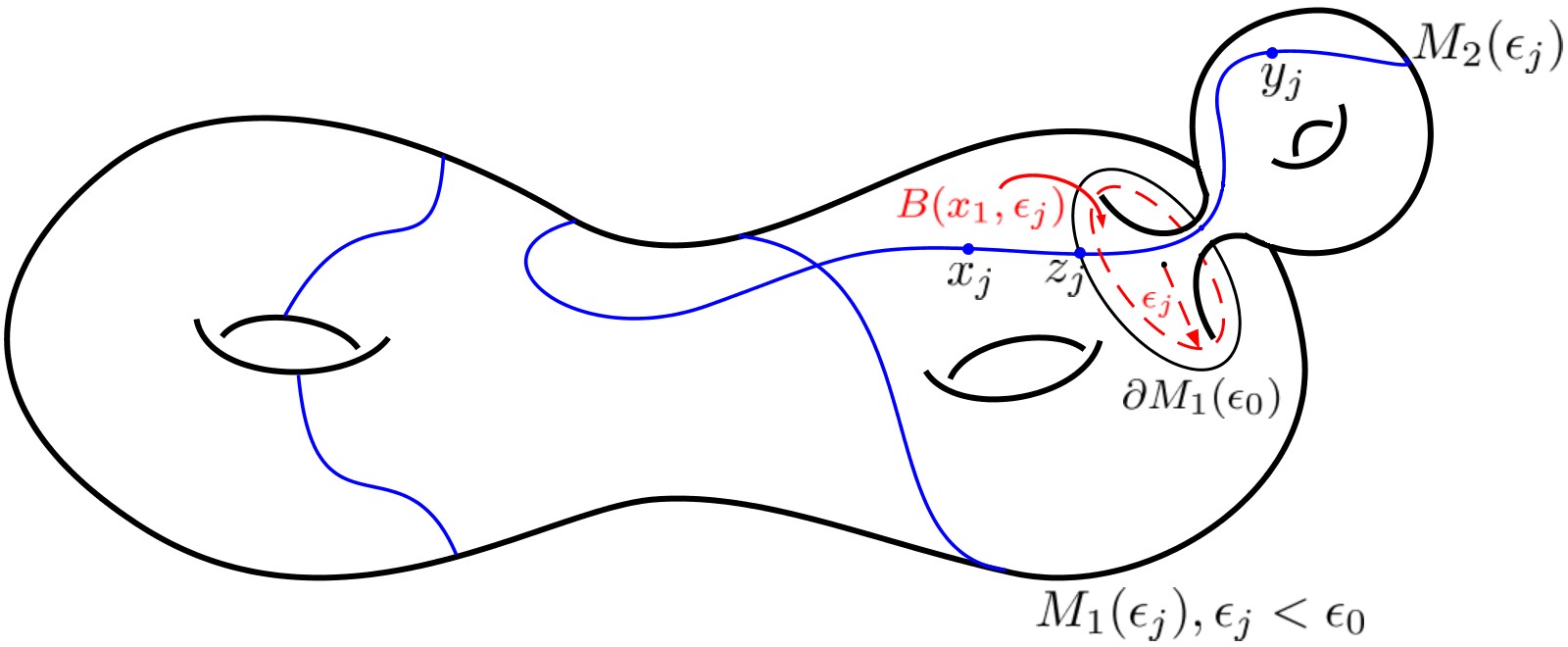}
     \vspace{2mm}
     
            {\small Figure 4}
    \end{center}

     \vspace{5mm}
Choosing a convergent subsequence of $\{z_n\}$ (and still calling it   $\{z_n\}$ with mild abuse of notation), we find a limit $z\in  \pa M_1(\epsilon_0)$, which satisfies that that $z\in \NNN_k(M_1)$ (by Lemma \ref{lem:lim_M_1}).

This implies that the intersection $\NNN_k(M_1)\cap \pa M_1(\epsilon_0)$ is non-empty. But we have assumed that the nodal set $\NNN_k(M_1)$ is at a distance $d$ from the part of the boundary $\pa B(x_1, \epsilon_0)$ of $M_1(\epsilon_0)$. This contradicts the choice of $\epsilon_0$. This shows that there cannot be any such subsequence $\{j\}\subset \{i\}$ i.e. there exists some $N\in \NN$ such that $\NNN_k(M_{\epsilon_i}) \subset M_1(\epsilon_0)$ for all $i>N$. 

Case II: If there is no infinite sequence of $\{x_i\}\in \NNN_k(M_{\epsilon_i})\cap M_1(\epsilon_0)$, then there exists a number $N\in \NN$ such that for any $i>N$, $\NNN_k(M_{\epsilon_i})\subset M_2(\epsilon)$. By Courant's theorem, $\NNN_k(M_{\epsilon_i})$ divides $M_{\epsilon_i}$ into at most $k$ parts. Define 
$$M_{\epsilon_i}^+=\{x\in M_{\epsilon_i}: f_k^{\epsilon_i}>0\} \text{ and  } M_{\epsilon_i}^-=\{x\in M_{\epsilon_i}: f_k^{\epsilon_i}<0\}.$$

Then $M_{\epsilon_i}= M_{\epsilon_i}^+ \sqcup \NNN_k(M_{\epsilon_i}) \sqcup M_{\epsilon_i}^-$ and $M_1(\epsilon_0)\subset M_{\epsilon_i}^+$ or $M_{\epsilon_i}^-$ for all $i>N$. The convergence $f_k^{\epsilon_i}|_{M_1(\epsilon_i)}:=f_k^{1, \epsilon_i}\to f_k$ in $C^0(M_1(\epsilon_0))$ implies that $f_k$ has to be either non-negative or non-positive on $M_1(\epsilon_0)$. But that cannot happen because $\NNN_k(M_1)$ is contained in $M_1(\epsilon_0)$ (by our assumption). So Case II can never happen.
\end{proof}

Lastly, to get a quantitative estimate on the order of the perturbation, we use Weyl's law in conjunction with Proposition \ref{thm:nod_wav_dense}. Recall that 
Weyl's law states that if $M$ is a compact connected oriented Riemannian manifold then
\begin{equation}\label{Theo: Weyl's law}
N(\lambda) = \frac{|B_n|}{(2\pi)^n}|M|\lambda^{n/2} + O(\lambda^{\frac{n - 1}{2}}),
\end{equation}
where $N(\lambda)$ is the number of eigenvalues of $-\Delta_M$ which are less than or equal to $\lambda$, and $|B_n|$ is the volume of the unit $n-$ball $B(0, 1)\subset \RR^n$. We have already shown that the nodal set $\NNN_k(M_{\epsilon})$ has null intersection with $M_2(\epsilon)$ for small enough $\epsilon$. However, $\NNN_k(M_{\epsilon})$ is also wavelength dense in $M_\epsilon$, which means that every point in $M_2(\epsilon)$ is at most at distance $C\lambda_k^{-1/2}$ from $\pa M_1(\epsilon)$. 
Also, using the triangle inequality, we see that there exists at least one point $\tilde{x} \in M_2(\epsilon)$ such that $d(\tilde{x}, \pa M_1(\epsilon)) \geq \frac{\epsilon}{2}\left(D - \tilde{D}\right)$, which implies that $\frac{\epsilon}{2}\left(D - \tilde{D}\right) \leq C\lambda_k^{-1/2}$. 
Now one can easily calculate that the constant $c$ in Remark \ref{rem:2.2}  comes from the  constant $C$ in Proposition \ref{thm:nod_wav_dense} via the application of Weyl's law:
\begin{equation}\label{eq:const_explic}
    c = \frac{C}{(D - \tilde{D})\pi}\left( |B_n||M_{\epsilon_0}|\right)^{1/n},
\end{equation}
where $B_n$ is the unit ball in $\RR^n$, $D, \tilde{D}$ are  the diameters of $M_2(1)$ and the metric sphere $\pa B(x_2, 1) \subseteq M_2$. Note that we can assume without loss of generality that $D > \tilde{D}$, as we can always start the construction in (\ref{eq:Riem_met_def}) with $M_2(\eta)$ for small enough $\eta$.

\begin{remark}
We observe that the proof above gives us slightly more, namely it gives us an analogous conclusion for level sets at small enough levels. Assume that the $\alpha$-level set of $f_k$ 
(that is, the set $f_k^{-1}(\alpha)$) does not intersect the geodesic ball $B(x_1, \epsilon_0)$, which is chosen at the beginning of the proof. Then the conclusion is that for small enough $\epsilon$, the $\alpha$-level set of $f^\epsilon_k$ 
lies completely inside $M_1$. The proof naturally splits into cases outlined above and follows the same steps. The proofs for the cases $(C)$ and $(D1)$ are 
completely similar, observing that the $C^\infty$-convergence of eigenfunctions away from $x_1$ gives convergence of level sets. On the other hand, the case (D2) follows from the fact that when $\epsilon$ is small enough, the $\alpha$-level set cannot fully lie inside $M_2(\epsilon)$. 
The latter follows since there is no $\alpha$-level set of $f^\epsilon_k$ inside $M_1(\epsilon) \setminus M_1(\epsilon_0)$, and hence $f_k^\epsilon$ has to grow arbitrarily fast around near $\pa M_1(\epsilon)$, making $|\nabla f^\epsilon_k |$ arbitrarily large.
\end{remark}

\section{A few applications}\label{sec:4}
\subsection{Payne property}
To illustrate an example, we now argue how we can get a perturbed version of Theorem 1.1 of \cite{M}. This will follow from an application of Lemmas \ref{lem:lim_M_1} and \ref{lem:nod_set_eventually}. Let $\Omega \subset \RR^n$ be a convex domain with smooth boundary which is perturbed to make a one-parameter family of (not necessarily convex) domains $\Omega_\epsilon := \Omega \# \epsilon\Omega'$, 
via the construction outlined after Corollary \ref{cor:Lewy}. We are interested in seeing whether the nodal set of the second Dirichlet eigenfunction for $\Omega_\epsilon$ intersects $\pa\Omega_\epsilon$ for $\epsilon$ small enough. 
For ease of writing, let us call this the ``Payne property''.

\begin{proposition}\label{cor:Melas}
Suppose $\Omega \subset \RR^n$ has the Payne property. Then there exists an $\epsilon_0 > 0$ such that 
$\Omega_\epsilon$ 
has the Payne property for all $\epsilon \leq \epsilon_0$.
\end{proposition}
\begin{proof}
Now, by Lemma \ref{lem:nod_set_eventually}, we know that $\NNN_2^\epsilon$ is eventually inside $\Omega$. By a straightforward topological argument, if  $\NNN_2^\epsilon$ does not intersect $\pa\Omega$, then it is an embedded hypersurface with possible ``lower dimensional'' singularities. 
By precompactness in Hausdorff metric (see, for example Theorem 2.2.25 of \cite{HP}), 
one can extract a subsequence called  $\NNN_2^{\epsilon_i}$, 
which converges to a set $X \subset \Omega$ 
in the Hausdorff metric. By Lemma \ref{lem:lim_M_1}, we already know that $X \subset \NNN_2(\Omega)$. 
It follows that for $i$ large enough, $\NNN_2^{\epsilon_i}$ is within any $\delta$-tubular neighbourhood of $\NNN_2(\Omega)$. 
This means that one of nodal domains of the second Dirichlet eigenfunction of $\Omega$ is within any $\delta$-tubular neighbourhood of $\NNN_2(\Omega)$ and the volume of such a tubular neighbourhood is going to $0$ as $\delta \searrow 0$. 
This will contradict the Faber-Krahn inequality (or the inner radius estimate for the second nodal domain of $\Omega$), and imply that for large enough $i$, $\NNN_2^{\epsilon_i}$ intersects the boundary. 
Moreover, if $\NNN_2$ is a submanifold, then using Thom's isotopy theorem (see \cite{AR}, Section 20.2) one can conclude that for large enough $i$, $\NNN_2^{\epsilon_i}$ is diffeomorphic to $\NNN_2(\Omega)$. This is precisely the case in dimension $2$, by Theorem 1.1 of \cite{M}. 
\end{proof}
In fact, as a further application of main ideas involved in proving Lemma \ref{lem:lim_M_1}, we state the following Proposition, which complements recent results (see for example, \cite{Ki} and references therein):
\begin{proposition}\label{cor:del_hol_cap}
Let $\Omega \subset \RR^n$ be a domain satisfying the Payne property. Pick $l$ points $x_1, x_2,..., x_l$ $ \in \Omega$ lying outside $\NNN_2(\Omega)$ and call $\Omega(\epsilon) = \Omega \setminus \cup_{i = 1}^l B(x_i, \epsilon)$. 
Then for $\epsilon$ small enough, $\Omega(\epsilon)$ satisfies the Payne property. 
\end{proposition}
\begin{proof}
If $\lambda_2$ is simple, and the capacity of the deficiency $\Omega \setminus \Omega(\epsilon)$ is going to $0$, we have that $\lambda_2(\epsilon) \searrow \lambda_2$, and a $C^0$ convergence of $f_2^\epsilon \to f_2$ in compact subsets of $\Omega \setminus \{ x_1,...., x_l\}$. A combination of Lemma \ref{lem:lim_M_1} and the argument of Proposition \ref{cor:Melas} now gives the result. 

Even if the Dirichlet Laplace spectrum of $\Delta$ is not simple, it follows from \cite{O, RT} that $\lambda_j(\epsilon) \to \lambda_j$, where $\lambda_j(\epsilon) = \lambda_j(\Omega(\epsilon))$ and $\lambda_j = \lambda_j(\Omega)$. We will use this below.

To take care of the case where $\lambda_2$ is repeated, we will adapt the ideas of Theorem 1.1 of \cite{T} to our setting, but here the situation is less complicated, as all the domains $\Omega(\epsilon)$ are sitting inside $\Omega$, so one can work directly with Sobolev spaces which are defined on the whole domain $\Omega$. Take an orthonormal family of Dirichlet eigenfunctions $f_1^{\epsilon_k},..., f_d^{\epsilon_k}$ of the Dirichlet Laplacian $\Delta_{\epsilon_k}$ of $\Omega(\epsilon_k)$. Since $\| f_j^{\epsilon_k} \|_{H_0^1(\Omega)}$ is bounded, we have a weakly-$H^1_0$ convergent subsequence (still called $f_j^{\epsilon_k}$ with minor abuse of notation), such that $f_j^{\epsilon_k} \to f_j \in H^1_0(\Omega)$. 
By compact Rellich embedding (up to a subsequence), the above convergence is strong in $L^2(\Omega)$.

Let $B : H_0^1(\Omega) \times H_0^1(\Omega) \to \RR$ be the bilinear form associated to $\Delta$, given by (for $u, v \in C^\infty_c(\Omega)$)
$$
B(u, v) := \int_\Omega \nabla u\nabla v.
$$
Let $\psi \in C^\infty_c(\Omega \setminus \{x_1,...,x_l\})$ be a test function, the latter set being dense in $H^1_0(\Omega)$. Now  we have 
\begin{align}\label{eq:Ozawa}
    B(f_j, \psi) & = \int_\Omega \nabla f_j \nabla \psi = \lim_{k \to \infty} \int_{\Omega(\epsilon_k)} \nabla f_j^{\epsilon_k}\nabla \psi \nonumber\\
    & = \lim_{k \to \infty} \int_{\Omega(\epsilon_k)} -\Delta_{\epsilon_k} f_j^{\epsilon_k}\psi + \lim_{k \to \infty} \int_{\pa\Omega(\epsilon_k)} (\pa_n f_j^{\epsilon_k}) \psi \nonumber \\
    & = \lim_{k \to \infty} \lambda_j(\epsilon_k)\int_{\Omega(\epsilon_k)} f_j^{\epsilon_k}\psi = \lambda_j \int_{\Omega} f_j\psi,
\end{align}
which implies that $f_j$ weakly solves the eigenequation 
\beq\label{eq:eigeneq}
-\Delta f_j = \lambda_j f_j.
\eeq
Observe that in (\ref{eq:Ozawa}), we have used Ozawa's result on the convergence of eigenvalues. By standard elliptic regularity results, $f_j$ solves (\ref{eq:eigeneq}) strongly, and is a $C^\infty$ Dirichlet eigenfunction of $\Delta$. This gives us $C^0$-convergence in compact subsets of $\Omega \setminus \{x_1,..., x_l\}$.

We will now check that orthonormality is preserved in the limiting process. That is, if $(f_i^{\epsilon_k}, f_j^{\epsilon_k}) = 0$, then $(f_i, f_j) = 0$. Since the convergence $f_j^{\epsilon_k} \to f_j$ is strong in $L^2$, we have 
\begin{align*}
    (f_i, f_j) & = (f_i - f_i^{\epsilon_k}, f_j - f_j^{\epsilon_k}) + (f_i^{\epsilon_k}, f_j - f_j^{\epsilon_k}) + (f_i, f_j^{\epsilon_k}).
\end{align*}
By Cauchy-Schwarz, the first two terms on the right are going to $0$. Also, 
$$ 
\lim_{\epsilon_k \to 0}(f_i, f_j^{\epsilon_k}) = \lim_{\epsilon_k \to 0} (f_i^{\epsilon_k}, f_j^{\epsilon_k}) = 0.
$$
This gives us the proof, as we have shown that eigenfunctions are converging to eigenfunctions in the limit,  orthonormality is preserved, and lastly, every eigenfunction in the limit domain, is the limit of some converging eigenfunction branch. For our application, we only need the convergence result for $j = 2$.
\end{proof}
It would not be difficult to generalize our result to reasonable domains $\Omega(\epsilon)$ such that $\capacity(\Omega \setminus \Omega(\epsilon)) \searrow 0$. Observe that the counterexample to the Payne property given in \cite{HHN} has a non-zero capacity obstacle that is removed from the domain. It would also be interesting to vary $l$ depending on $\epsilon$ in the spirit of homogenization problems, as in \cite{RT}.

\subsection{Proof of Corollary \ref{cor:Lewy}}
Now we give an outline of the proof of Corollary \ref{cor:Lewy} as a straightforward application of Theorem \ref{thm:main_thm}. Lewy proved a very interesting lower estimate on the number of nodal domains of spherical harmonics. The following is a combination of Theorems 1 and 2 from \cite{L}:
\begin{theorem}[Lewy]\label{thm:Lewy}
Let $k \in \NN$ be odd. Then there is a spherical harmonic of degree $k$ with exactly two nodal domains. Let $k \in \NN$ be even. Then there is a spherical harmonic of degree $k$ with exactly three nodal domains.
\end{theorem}
 Note that, when a spherical harmonic has exactly two nodal domains, the nodal line is diffeomorphic to a circle and when a spherical harmonic has exactly three nodal domains, the nodal line is the union of two disjoint embedded circles (see for example, Theorem 2.5(ii) of \cite{Ch}). An application of Theorem \ref{thm:main_thm} together with the ideas of Proposition \ref{cor:Melas} now allows us to construct such eigenfunctions with minimal nodal domains on any closed genus $\gamma$ surface, which has been recorded in Corollary \ref{cor:Lewy}. 
Note that, however, this metric has positive curvature for the most part, and a small region of high negative curvature. It would be interesting to investigate whether in addition one can force this metric to be of constant negative curvature. Maybe some ideas involving geometric flows might help (for example, as in \cite{Mu}).

\section{Acknowledgements} The authors are grateful to IIT Bombay for providing ideal working conditions. The second named author would like to thank the Council of Scientific and Industrial Research, India for funding which supported his research. The authors would like to acknowledge useful conversations with Bogdan Georgiev, Gopala Krishna Srinivasan and Harsha Hutridurga and the first author would like to thank Nitin Nitsure for asking him the question of sand patterns on symmetrical plates mentioned on page 2. 
Further, the authors are deeply indebted to Junya Takahashi for giving a detailed clarification of certain aspects of Theorem \ref{thm:dom_eigen_conv}. 


\end{document}